\documentclass[11pt]{amsart}
\usepackage{amsmath,amssymb,amsthm}
\usepackage{tikz}
\usetikzlibrary{decorations.markings}

\newtheorem{theorem}{Theorem}[section]

\newtheorem{lemma}[theorem]{Lemma}

\newcommand{\Aut}{\hbox{{\rm Aut}}}
\newcommand{\Sym}{{\rm Sym}}
\newcommand{\Alt}{{\rm Alt}}
\newcommand{\Hom}{\hbox{{\rm Hom}}}

\newcommand{\Cay}{\hbox{Cay}}
\newcommand{\Cov}{\hbox{Cov}}
\newcommand{\supp}{\hbox{Supp}}

\newcommand{\V}{{\rm{V}}}
\newcommand{\A}{\rm{A}}
\newcommand{\K}{\rm{K}}
\newcommand{\Q}{\rm{Q}}

\newcommand{\p}{\wp}

\newcommand{\tB}{\tilde{B}}
\newcommand{\tA}{\tilde{A}}

\newcommand{\NN}{\mathbb{N}}
\newcommand{\ZZ}{\mathbb{Z}}
\newcommand{\FF}{\mathbb{F}}

\newcommand{\MK}{{\mathcal{M}}}

\title{Three  local actions in $6$-valent arc-transitive graphs}

\author{Ademir Hujdurovi\'{c}}
\author{Primo\v{z} Poto\v{c}nik}
\author{Gabriel Verret}

\address{Ademir Hujdurovi\'{c}, University of Primorska, UP IAM, Muzejski trg 2, SI-6000 Koper, Slovenia.\newline
\indent Also affiliated with: University of Primorska, UP FAMNIT, Glagolja\v ska 8, SI-6000 Koper, Slovenia.
}
\email{ademir.hujdurovic@upr.si}	

\address{Primo\v{z} Poto\v{c}nik, Faculty of Mathematics and Physics, University of Ljubljana, Jadranska 21, SI-1000 Ljubljana, Slovenia.\newline
\indent Also affiliated with: Institute of Mathematics, Physics and Mechanics, Jadranska 19, SI-1000 Ljubljana, Slovenia.
}
\email{primoz.potocnik@fmf.uni-lj.si}

\address{Gabriel Verret\\
Department of Mathematics, The University of Auckland\\
Private Bag 92019, Auckland 1142, New Zealand.}
\email{g.verret@auckland.ac.nz}

\begin{document}

\begin{abstract}
It is known that there are precisely three transitive permutation groups of degree $6$ that admit an invariant partition with three parts of size $2$ such that the kernel of the action on the parts has order $4$; these groups are called $A_4(6)$, $S_4(6d)$ and $S_4(6c)$. For each $L\in \{A_4(6), S_4(6d), S_4(6c)\}$, we construct an infinite family of finite connected $6$-valent graphs $\{\Gamma_n\}_{n\in \NN}$ and arc-transitive groups $G_n \le \Aut(\Gamma_n)$ such that the permutation group  induced by the action of the vertex-stabiliser $(G_n)_v$ on the neighbourhood of a vertex $v$ is permutation isomorphic to $L$, and such that $|(G_n)_v|$ is exponential in $|\V(\Gamma_n)|$. These three groups were the only transitive permutation groups of degree at most $7$ for which the existence of such a family was undecided. In the process, we construct an infinite family of cubic $2$-arc-transitive graphs such that the dimension of the $1$-eigenspace over the field of order $2$ of the adjacency matrix of the graph grows linearly with the order of the graph.
\end{abstract}

\maketitle

\section{Introduction}

All the graphs in this paper are finite, connected, simple and undirected. An \emph{arc} of a graph $\Gamma$ is an ordered pair of adjacent vertices. The set of vertices of  $\Gamma$ is denoted by $\V(\Gamma)$ and the set of arcs by $\A(\Gamma)$. The {\em automorphism group $\Aut(\Gamma)$} of $\Gamma$ is the group of  all permutations  of   $\V(\Gamma)$ preserving $\A(\Gamma)$.  We say that $\Gamma$ is \emph{$G$-arc-transitive} if $G$ is a subgroup of $\Aut(\Gamma)$ acting transitively on $\A(\Gamma)$.

One of the central questions in the study of $G$-arc-transitive graphs is to find good upper bounds on the order of an arc-stabiliser $G_{uv}$ under certain hypotheses. This problem was first considered in Tutte's celebrated work~\cite{T48}, where it was proved that, in the $3$-valent case, $|G_{uv}|\leq 16$. 

Bounding the order of the arc-stabiliser plays a crucial role in many problems, such as when constructing complete lists of graphs of a prescribed symmetry type \cite{newlist,ConDob02,cubicSS,CubicCensus,HATcensus} (which have numerous applications themselves), proving asymptotic results regarding the number of graphs of a particular type \cite{twogroups,enum}, obtaining classification results see \cite{ConLiPot,ConNed,MorSVer,five}, or various other problems in algebraic graph theory and elsewhere \cite{caprace,lifts,pablo2,Tor}.

When bounding $|G_{uv}|$ in a family of $G$-arc-transitive graphs by a constant is not possible, it is still worthwhile to bound it by a suitable  function of the order of the graph. If such a function (even if not constant) grows slowly enough, many of the above mentioned applications are still possible. For example, it was proved in \cite{lost} that there exists a sublinear function $f(n)$ such that for every connected $G$-arc-transitive $4$-valent graph $\Gamma$, not belonging to a well-understood exceptional family, the inequality $|G_{uv}| \le f(|\V(\Gamma)|)$ holds. This fact was then used both to construct a complete list of all $4$-valent arc-transitive graphs on at most $640$ vertices (and also a complete list of all $3$-valent vertex-transitive graphs on at most $1280$ vertices \cite{CubicCensus,recipe}) as well as prove an interesting asymptotic result pertaining to the number of such graphs up to a given order~\cite{enum}.

The problem of bounding $|G_{uv}|$ in a family of $G$-arc-transitive graphs is typically considered in terms of the {\em local action}, which we now introduce. Let $\Gamma$ be a connected finite $G$-arc-transitive graph and let $G_v^{\Gamma(v)}$ denote the permutation group induced by the action of the vertex-stabiliser  $G_v$ on the neighbourhood $\Gamma(v)$ of a vertex $v$. Since $G$ is transitive on $\A(\Gamma)$ and $\V(\Gamma)$, the group $G_v^{\Gamma(v)}$ is transitive and (up to permutation isomorphism) independent of the choice of $v$. If $L$ is a permutation group which is permutationally equivalent to $G_v^{\Gamma(v)}$, we say that the pair $(\Gamma,G)$ is \emph{locally-$L$}, or that $(\Gamma,G)$ has \emph{local action} $L$.

A transitive group $L$ is called {\em graph-restrictive} \cite{Verret} if there exists a constant $c_L$ such that, for every locally-$L$ pair $(\Gamma,G)$ and $(u,v)\in\A(\Gamma)$, we have $|G_{uv}|\leq c_L$. Tutte's result~\cite{T48} can then be rephrased as saying that transitive groups of degree $3$ are graph-restrictive (with corresponding constant $16$),  while the famous Weiss conjecture \cite{Weiss} claims that every primitive permutation group is graph-restrictive; see \cite{restrictive} for a survey of results on graph-restrictiveness and \cite{caprace,GiuMor1,GiuMor2,pablo} for more recent results.

Given a transitive group $L$ that is not graph-restrictive, one is naturally led to wonder about the existence of a ``tame'' function $f_L\colon \NN \to \NN$ such that, for every locally-$L$ pair $(\Gamma,G)$, we have 
\begin{equation}
\label{eq:fL}
|G_{uv}| \le f_L(|\V(\Gamma)|).
\end{equation}
 It is an easy exercise to show that, for every $L$, there is an exponential function $f_L$ satisfying (\ref{eq:fL}); see \cite[Theorem 5]{frestrictive}, for example. On the other hand, if there exists an exponential function $g_L$ and an infinite family of locally-$L$ pairs $\{(\Gamma_n,G_n)\}_{n\in \NN}$ such that $|(G_n)_{uv}|\geq g_L(|\V(\Gamma_n)|)$, then a function $f_L$ satisfying (\ref{eq:fL}) cannot grow slower than every exponential function;
 in this case, we say that  $L$ is of {\em exponential type}.
 In a similar way, one can define other types. For example if in (\ref{eq:fL}) $f_L$ can be chosen to be
constant, then $L$ is of {\em constant type} (and $L$ is then graph-restrictive).
Similarly, if $L$ is not of constant type but (\ref{eq:fL}) is satisfied with $f_L(n) = n^\alpha$ , for some $\alpha >0$, then $L$ is of {\em polynomial type}. Finally, if $L$ is not of constant, polynomial or exponential type, then $L$ is of {\em subexponential type}.
Determining the type of transitive groups was the central topic of~\cite{frestrictive}.

While one can show that there exist permutation groups of constant type (that is, graph-restrictive groups),  of polynomial type  \cite{Verret2}, and of exponential type,  the existence of permutation groups of subexponential type is still an open problem \cite[Question~4]{frestrictive}. The smallest undecided cases are three permutation groups of degree six, which can be characterised as follows.

Up to permutation equivalence, there are $16$ transitive groups of degree $6$ (see~\cite{conway}). Of these, exactly three admit an invariant partition with three parts of size two such that the kernel of the action on the parts has order $4$. Using the terminology from~\cite{conway}, these are:
\begin{itemize}
\item $A_4(6)$, TransitiveGroup(6,4), $\Alt(4)$ acting on the cosets of $\langle (1~2)(3~4)\rangle$, 
\item $S_4(6d)$, TransitiveGroup(6,7), $\Sym(4)$ acting on the cosets of $\langle (1~2),(3~4)\rangle$, 
\item $S_4(6c)$, TransitiveGroup(6,8), $\Sym(4)$ acting on the cosets of $\langle (1~2~3~4)\rangle$.
\end{itemize}
For the groups $A_4(6)$, $S_4(6d)$ and $S_4(6c)$, the authors of~\cite{frestrictive} were unable to determine whether these groups were of subexponential or exponential type.
The main result of this paper is to solve this problem by proving the following theorem.

\begin{theorem}
\label{theo:main2}
Each of the permutation groups $A_4(6)$, $S_4(6d)$ and $S_4(6c)$ is of exponential type.
\end{theorem}

In particular, the question of existence of permutation groups of subexponential type is still open. Moreover, since the type of all transitive permutation groups of degree at most $7$ was determined in \cite{frestrictive} with the exception of the groups $A_4(6)$, $S_4(6d)$ and $S_4(6c)$, the next unresolved case occurs in degree $8$, where there are four groups of unknown type, with IDs (8,24), (8,32), (8,39) and (8,40) in the database of transitive groups in {\sc Magma}.

We prove Theorem~\ref{theo:main2} by explicitly constructing a family of pairs  $(\Gamma_n,G_n)$ with the appropriate local action and $|(G_n)_{uv}|$ exponential in $|\V(\Gamma_n)|$. In our construction, the graph $\Gamma_n$ is defined as the lexicographic product of a $3$-valent arc-transitive graph $\Lambda_n$ with an edgeless graph on two vertices. The graphs $\Lambda_n$ are chosen so that for $L\in\{A_4(6), S_4(6d), S_4(6c)\}$ one can find a subgroup $G_n\le \Aut(\Gamma_n)$ such that $(\Gamma_n,G_n)$ is locally-$L$ and moreover, such that $|(G_n)_{uv}|$ grows exponentially with $|\V(\Gamma_n)|$.

The key property of the graphs $\Lambda_n$ that allows us to do this is, perhaps surprisingly, that
they have large $1$-eigenspace over $\FF_2$. (Recall that the eigenspace of a graph is the eigenspace of its adjacency matrix.) To be more precise, we require that the dimension of the $1$-eigenspace of the graphs $\Lambda_n$ grows linearly with $|\V(\Lambda_n)|$.
We thus obtain the following result along the way.

\begin{theorem}\label{theo:main}
There exists an infinite family of connected $3$-valent $2$-arc-transitive graphs $\{\Lambda_n\}_{n\in \NN}$ such that the $1$-eigenspace of $\Lambda_n$ over $\FF_2$ has dimension at least $|\V(\Lambda_n)|/72$.
\end{theorem}

We would like to point out that the existence of the family $\{\Lambda_n\}_{n\in \NN}$ was strongly suggested by the census of cubic arc-transitive graphs \cite{ConDob02}.


We should also mention that the approach that we use in the proof of Theorem~\ref{theo:main2}
is similar to the one used in \cite{frestrictive}, where the $0$-eigenspaces of cubic $2$-arc-transitive graphs were used instead of the $1$-eigenspaces. It was shown in \cite{nullity} that the $0$-eigenspace in a $2$-arc-transitive cubic graph can have arbitrary large dimension, it is still an open question whether there exists an infinite family of such graphs such that the dimension is linear in the order of the graph. In Section~\ref{MKGraph} we define the graphs $\Lambda_n$ and prove Theorem~\ref{theo:main}.
Section~\ref{sec:main} is devoted to the proof of Theorem~\ref{theo:main2}.

\section{The M\"obius--Kantor graph and some of its arc-transitive covers}
\label{MKGraph}

In this section we construct a family of cubic $2$-arc-transitive graphs $\Lambda_n$ as covers of the M\"obius-Kantor graph and show that they have a large $1$-eigenspace over $\FF_2$.

Consider the group 
$$R=\langle a,b,c,z\mid 1=a^2=b^2=c^2=z^2=[a,z]=[b,z]=[c,z],[a,b]=[b,c]=[a,c]=z\rangle.$$
It not hard to see that $R$ is a group of order $16$. The Cayley graph $\MK=\Cay(R,\{a,b,c\})$ is a connected $3$-valent vertex-transitive graph called the \emph{M\"obius--Kantor graph}. (See Figure~\ref{MKpic}, where $a$-edges are dotted, $b$-edges are  black, and $c$-edges are  gray. Arrows and edge labels can be ignored for now.) It is obvious from the given presentation of $R$ that any permutation of $\{a,b,c\}$ induces an automorphism of $R$. It follows that $\MK$ admits a group of automorphisms $B$ isomorphic to $R\rtimes \Sym(3)$. (In fact, $B$ is the full automorphism group of $\MK$, but we will not use this fact.) Note that $B$ is $2$-arc-regular and contains an arc-regular subgroup $A$ of the form $R\rtimes \ZZ_3$.

\tikzstyle{vertex}=[circle, draw, fill=white!50,inner sep=0pt, minimum width=22pt]
\tikzstyle{lab}=[node distance=1.3cm,bend angle=45,auto]
\tikzstyle{edgeA} = [draw,dotted, line width=3pt,-,black]
\tikzstyle{edgeB} = [draw,line width=3pt,-,gray]
\tikzstyle{edgeC} = [draw,line width=3pt,-,black]

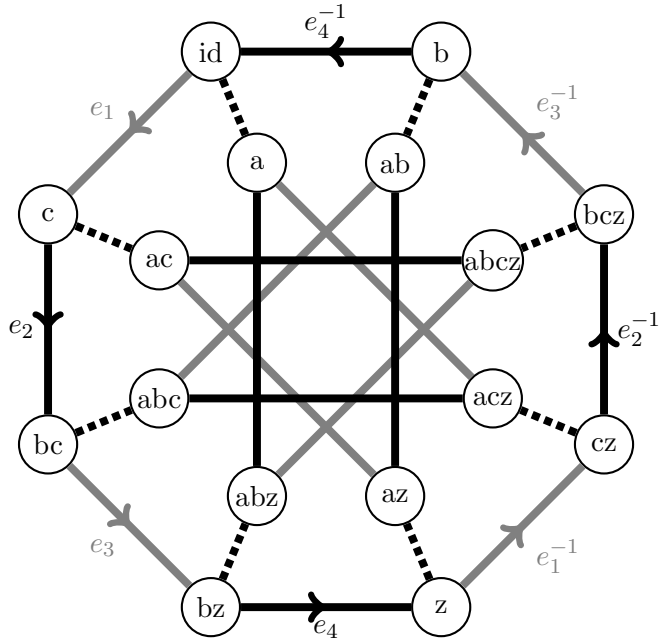
\begin{figure}[hh]
\centering
\begin{tikzpicture}[thick,scale=0.8,decoration={markings, mark=at position 0.5 with {\arrow{>}}}]
\node[vertex]  (abcz) at (22.5:3){abcz};
\node[vertex]  (ab) at (67.5:3){ab};
\node[vertex]  (a) at (112.5:3){a};
\node[vertex]  (ac) at (157.5:3){ac};
\node[vertex]  (abc) at (202.5:3){abc};
\node[vertex]  (abz) at (247.5:3){abz};
\node[vertex]  (az) at (292.5:3){az};
\node[vertex]  (acz) at (337.5:3){acz};
\node[vertex]  (bcz) at (22.5:5){bcz};
\node[vertex]  (b) at (67.5:5){b};
\node[vertex]  (1) at (112.5:5){id};
\node[vertex]  (c) at (157.5:5){c};
\node[vertex]  (bc) at (202.5:5){bc};
\node[vertex]  (bz) at (247.5:5){bz};
\node[vertex]  (z) at (292.5:5){z};
\node[vertex]  (cz) at (337.5:5){cz};

\path (1) edge[edgeA] (a);
\path (b) edge[edgeA]  (ab) ;
\path (c) edge[edgeA] (ac);
\path (z) edge[edgeA] (az);
\path (bc) edge[edgeA] (abc);
\path (bz) edge[edgeA] (abz);
\path (cz) edge[edgeA] (acz);
\path (bcz) edge[edgeA](abcz); 
\path (a) edge[edgeB] (acz);
\path (abz) edge[edgeB]  (abcz) ;
\path (az) edge[edgeB] (ac);
\path (abc) edge[edgeB] (ab);
\path (az) edge[edgeC] (ab);
\path (a) edge[edgeC]  (abz) ;
\path (ac) edge[edgeC] (abcz);
\path (abc) edge[edgeC] (acz);

\path (1) edge[edgeB,postaction={decorate}] node[lab] [swap] {$e_1$} (c);    
\path (c) edge[edgeC,postaction={decorate}] node[lab]  [swap]{$e_2$} (bc); 
\path (bc) edge[edgeB,postaction={decorate}] node[lab]  [swap]{$e_3$} (bz); 
\path (bz) edge[edgeC,postaction={decorate}] node[lab] [swap] {$e_4$} (z);   
\path (z) edge[edgeB,postaction={decorate}] node[lab]  [swap]{$e_1^{-1}$} (cz); 
\path (cz) edge[edgeC,postaction={decorate}] node[lab]  [swap]{$e_2^{-1}$} (bcz);  
\path (bcz) edge[edgeB,postaction={decorate}] node[lab]  [swap]{$e_3^{-1}$} (b); 
\path (b) edge[edgeC,postaction={decorate}] node[lab]  [swap]{$e_4^{-1}$} (1);    
\end{tikzpicture}
\caption{M\"obius--Kantor graph}\label{MKpic}
\end{figure}

We will now define the graphs $\Lambda_n$ as covering graphs over $\MK$. But first we need to introduce a few notions from the theory of covering graphs. Here we only give a very brief account of the subject and refer the reader to \cite{elemab} for more comprehensive treatment.

Let $\Gamma$ be a connected graph and $N$ a group. Assign to each arc $(u,v)$ of $\Gamma$ a {\em voltage} $\zeta(u,v) \in N$ such that $\zeta(v,u) =  (\zeta(u,v))^{-1}$. Let $\Cov(\Gamma;\zeta)$ be the {\em derived covering graph} with vertex set $\V(\Gamma) \times N$ and adjacency relation defined by $(u,a) \sim (v,a\, \zeta(u,v))$, where $u \sim v$ in $\Gamma$. For a fixed vertex $v\in \V(\Gamma)$, the set $\{v\} \times N$ is called the {\em fibre} above $v$. If $\zeta$ maps every arc of some spanning tree of $\Gamma$ to the trivial element of $N$, then we say that $\zeta$ is {\em normalised}; in this case $\Cov(\Gamma;\zeta)$ is connected if and only if the image of $\zeta$ generates $N$. If $C=v_0v_1\ldots v_nv_0$ is a closed walk of $\Gamma$, then we let $\zeta(C)= \zeta(v_nv_0)\zeta(v_{n-1}v_n) \cdots \zeta(v_0v_1)$.

An automorphism of $\tilde{g}\in\Aut(\Cov(\Gamma;\zeta))$ which preserves the partition of the vertices into fibres clearly induces an automorphism $g$ of $\Gamma$; in this case, we say that $\tilde{g}$ {\em projects}  to $g$ and  that $g$ {\em lifts} to $\tilde{g}$. If every automorphism of a group $G\le \Aut(\Gamma)$ lifts, then the set $\tilde{G}$ of lifts of all the elements of $G$ forms a group, called the {\em lift of $G$}. It is well known that in this case the group $G$ is vertex-transitive (or edge-transitive or $s$-arc-transitive) if and only if $\tilde{G}$ is vertex-transitive (or edge-transitive or $s$-arc-transitive, respectively).

We are now ready to define the graphs $\Lambda_n$.
Let $n$ be a positive integer and let $N=\ZZ_n^4=\langle e_1,e_2,e_3,e_4\rangle$. Let $\zeta_n\colon\ A(\MK) \to N$ be the voltage assignment given in Figure~\ref{MKpic} (where unlabelled arcs get trivial voltage). Let $\Lambda_n$ be the derived covering graph with respect to $\zeta_n$.

Note that since $\zeta_n$ is normalised and the image of $\zeta_n$ generates $N$, the graph $\Lambda_n$ is connected.  In the case when $n$ is prime, $\zeta_n$ was described in \cite{gp83} and corresponds to the subspace  ${\mathcal{L}}_R(i,-i) \le \Hom(\MK;\ZZ_n)$ generated by the vectors $\langle c_1, c_2, c_3, c_4\rangle$ described in~\cite[page 2168]{gp83}, or equivalently in the case where $n\equiv 1 \pmod 4$, by the vectors $b_1, b_2, b_3, b_4$~\cite[page 2165]{gp83}.

In the case of $n$ being prime, it was proved in \cite{gp83} that $B$ lifts along the covering projection $\Lambda_n \to \MK$, and the argument there generalises easily to an arbitrary positive integer $n$. This can also be proved using the following direct approach: In view of \cite[Proposition 5.1]{elemab}, an automorphism $g \in \Aut(\MK)$ lifts along the covering projection $\Lambda_n \to \MK$ if and only if there exists an automorphism $g^\# \in \Aut(N)$ such that $\zeta_n(C)^{g^\#} = \zeta_n(C^g)$ for every closed walk $C$, or equivalently, for every generator $C$ of a fixed generating set of the cycle space of $\MK$. Recall that a generating set of the cycle space can be constructed by choosing a spanning tree $T$ and then, for every edge $uv$ not in $T$, taking the cycle  whose only arc not in $T$ is $(u,v)$. In our case, we have chosen the spanning tree which consists of all the $a$-edges (the dotted edges) and all the edges of the inner $8$-cycle except the edge $\{ac, abcz\}$ (see Figure~\ref{MKpic}), and then found an appropriate automorphism $g^\#$ for every generator $g$ of $B$. For example, if $g$ is the automorphism mapping the vertices of $\MK$ according to the rule $v\mapsto v\cdot z$ (which in Figure~\ref{MKpic} corresponds to the reflection  through the central point of the figure), then the corresponding $g^\#$  maps every $x\in N$ to its inverse $x^{-1}$. This calculation was done with the help of {\sc Magma} \cite{mag}.

Let $\tB \le \Aut(\Lambda_n)$ denote the lift of $B$ and let $\tA$ be the lift of $A$. Note that $\tB$ is $2$-arc-regular on $\Lambda_n$, and $\tA$ is $1$-arc-regular. In particular, $|\tB|=2|\tA|=6\cdot 16n^4$.

\section{Proof of Theorem~\ref{theo:main}}\label{sec:main}

Let $n\geq 3$ be a positive integer and let $\Lambda=\Lambda_n$. Further, let $\FF_2^{\V(\Lambda)}$ denote the $\FF_2$-vector space of all functions from $\V(\Lambda)$ to $\FF_2$. An element $x\in \FF_2^{\V(\Lambda)}$ is a \emph{$1$-eigenvector} for $\Lambda$ if and only if $\sum_{u\in\Gamma(v)} x(u)=x(v)$ for every $v\in \V(\Gamma)$.

Our first goal is to exhibit a $1$-eigenvector for $\Lambda_n$ whose support
(the number of vertices mapped to $1$) is bounded (that is, does not grow with $n$). Finding such an eigenvector  proved to be a surprisingly difficult task, which we could not have performed without the considerable help of a computer. Let $S_1\subseteq \V(\Lambda)=\V(\MK)\times\ZZ_n^4$ be defined as follows:
\begin{equation*}
\begin{split}
S_1=&\{
(\hbox{id}, 1, 0, 1, 1 ),
(\hbox{id}, 1, 1, 1, 1 ),
(\hbox{id}, 1, 1, 2, 1 ),
(\hbox{id}, 1, 2, 2, 1 ),
(a, 0, 0, 1, 2 ),
(a, 0, 1, 2, 2 ),
(a, 1, 1, 1, 1 ),\\
&(a, 1, 2, 2, 1 ),
(ab, 0, 0, 1, 2 ),
(ab, 1, 1, 1, 1 ),
(ab, 1, 1, 1, 2 ),
(ab, 1, 1, 2, 2 ),
(ab, 1, 2, 2, 2 ),
(ab, 2, 1, 1, 1 ),\\
&(abc, 0, 1, 1, 2 ),
(abc, 1, 1, 1, 1 ),
(abc, 1, 2, 2, 2 ),
(abc, 2, 2, 2, 1 ),
(abcz, 0, 0, 1, 2 ),
(abcz, 1, 0, 1, 1 ),\\
&(abcz, 1, 1, 2, 2 ),
(abcz, 2, 1, 2, 1 ),
(abz, 0, 1, 2, 2 ),
(abz, 1, 0, 1, 1 ),
(abz, 1, 1, 1, 1 ),
(abz, 1, 1, 2, 1 ),\\
&(abz, 1, 1, 2, 2 ),
(abz, 2, 2, 2, 1 ),
(ac, 0, 0, 1, 2 ),
(ac, 1, 1, 1, 1 ),
(ac, 1, 1, 2, 2 ),
(ac, 2, 2, 2, 1 ),\\
&(acz, 0, 0, 1, 2 ),
(acz, 1, 1, 1, 1 ),
(acz, 1, 1, 2, 2 ),
(acz, 2, 2, 2, 1 ),
(az, 1, 0, 1, 2 ),
(az, 1, 1, 2, 2 ),\\
&(az, 2, 1, 1, 1 ),
(az, 2, 2, 2, 1 ),
(b, 0, 0, 1, 2 ),
(b, 0, 1, 1, 2 ),
(b, 1, 0, 1, 2 ),
(b, 1, 1, 1, 2 ),
(bc, 0, 1, 1, 2 ),\\
&(bc, 1, 1, 1, 1 ),
(bc, 1, 1, 1, 2 ),
(bc, 2, 1, 1, 1 ),
(bcz, 0, 1, 2, 2 ),
(bcz, 1, 1, 2, 1 ),
(bcz, 1, 1, 2, 2 ),\\
&(bcz, 2, 1, 2, 1 ),
(bz, 1, 1, 2, 1 ),
(bz, 1, 2, 2, 1 ),
(bz, 2, 1, 2, 1 ),
(bz, 2, 2, 2, 1 ),
(c, 1, 0, 1, 1 ),
(c, 1, 0, 1, 2 ),\\
&(c, 1, 1, 1, 1 ),
(c, 1, 1, 2, 2 ),
(c, 2, 1, 1, 1 ),
(c, 2, 1, 2, 1 ),
(cz, 0, 1, 1, 2 ),
(cz, 0, 1, 2, 2 ),
(cz, 1, 1, 1, 1 ),\\
&(cz, 1, 1, 2, 2 ),
(cz, 1, 2, 2, 1 ),
(cz, 1, 2, 2, 2 ),
(z, 1, 0, 1, 2 ),
(z, 1, 1, 1, 2 ),
(z, 1, 1, 2, 2 ),
(z, 1, 2, 2, 2 )
\}.
\end{split}
\end{equation*}

As alluded to earlier, $S_1$ was found with the help of a computer. We do not have a ``natural'' description of this set or of another set with the desired properties. See Figure~\ref{fig:72} for a drawing of the induced subgraph of $\Lambda$ on $S_1\cup \Lambda(S_1)$. Vertices are colored black if they are in $S_1$, white otherwise. Dashed edges belong to fibres over edges of $\MK$ with trivial voltage.  Gray and black solid edges are in fibres over $b$-edges and $c$-edges with non-trivial voltages, respectively.

\begin{figure}
\begin{center}
\includegraphics[width=0.95\textwidth]{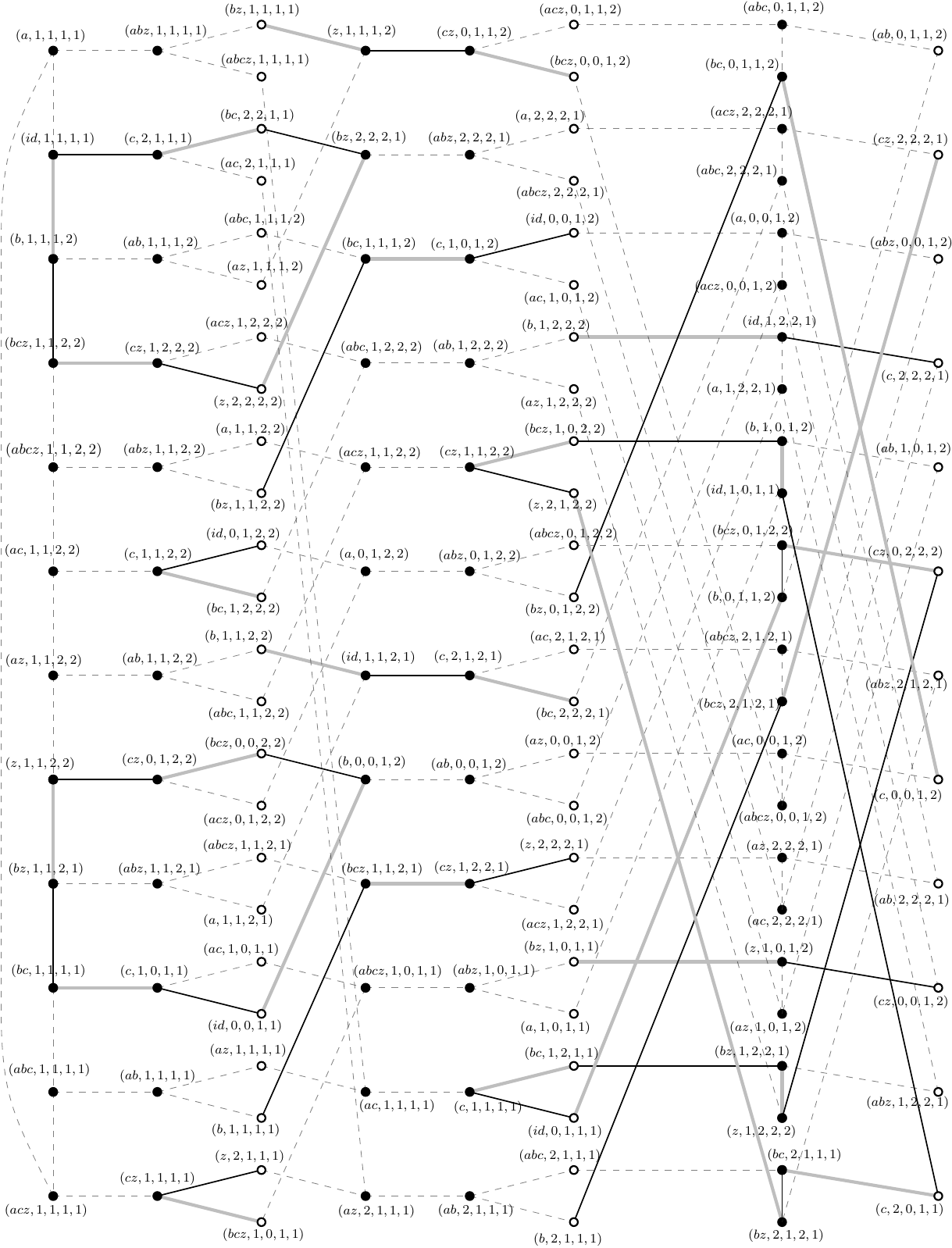} 
\caption{$\Lambda[S_1\cup \Lambda(S_1)]$. }\label{fig:72}
\end{center}
\end{figure}

One can check (using Figure~\ref{fig:72} or by computer) that every element of $S_1$ has an odd number of neighbours in $S_1$. (In Figure~\ref{fig:72}, this is equivalent to every black vertex having an odd number of black neighbours.) Similarly, one can check that every vertex not in $S_1$ has an even number of neighbours in $S_1$. (Vertices that do not appear in Figure~\ref{fig:72} have no black neighbours. This is thus equivalent to the fact that, in  Figure~\ref{fig:72}, every white vertex has an even number of black neighbours.)  In Section~\ref{Appendix}, we provide {\sc magma} code that checks this claim.

Let $x_1\in \FF_2^{\V(\Lambda)}$ satisfying $x_1(v)=1\Leftrightarrow v\in S_1$. It follows from the paragraph above that $x_1$ is a $1$-eigenvector for $\Lambda$. Let $\{x_1,\ldots,x_k\}$ be a linearly independent set of $1$-eigenvectors for $\Lambda$, and let $S_i$ be the support of $x_i$. Suppose that there exists  $v\in V(\Lambda)\setminus (\bigcup_{i=1}^k S_i)$. Since $\Lambda$ is vertex-transitive, there exists $g_{k+1}\in \Aut(\Lambda)$ such that $v\in S_1^{g_{k+1}}$. Let $x_{k+1}=x_1^{g_{k+1}}$. Since ${g_{k+1}}\in \Aut(\Lambda)$, $x_{k+1}$ is an $1$-eigenvector for $\Lambda$.    Moreover, $x_{k+1}$ has support $S_1^{g_{k+1}}$ which contains $v$. Since $\{x_1,\ldots,x_k\}$ is linearly independent, so is $\{x_1,\ldots,x_k,x_{k+1}\}$.

Starting with $\{x_1\}$ and repeatedly applying the procedure described in the previous paragraph, we obtain a linearly independent set of $1$-eigenvectors $\{x_1,\ldots,x_t\}$ such that $\V(\Lambda)=\bigcup_{i=1}^t \supp(x_i)$ and, for each $i=1,\ldots,t$,  $x_i=x_1^{g_i}$ for some $g_i\in\Aut(\Lambda)$. It follows that $|\supp(x_i)|=72$, $|\V(\Lambda)|\leq \sum_{i=1}^t |\supp(x_i)|=72t$ and thus $\{x_1,\ldots,x_t\}$ is a linearly independent set of $1$-eigenvectors for $\Lambda$ of size at least $|\V(\Lambda)|/72$. This concludes the proof of Theorem~\ref{theo:main}.


\section{Proof of Theorem~\ref{theo:main2}}
Let $n\geq 1$, let $\Lambda=\Lambda_n$ and let $\Gamma$ be the lexicographic product $\Lambda[\overline{K_2}]$. In other words, $\V(\Gamma)=\V(\Lambda)\times \FF_2$ with $(u,x)$ adjacent to $(v,y)$ in $\Gamma$ if and only if $u$ is adjacent to $v$ in $\Lambda$. Since $\Lambda$ is a connected $3$-valent graph, $\Gamma$ is a connected $6$-valent graph.

Note that $\Aut(\Lambda)$ has a natural action as a group of automorphisms of $\Gamma$ (by acting on the first coordinate of vertices of $\Gamma$ while fixing the second coordinate). Similarly,  $\FF_2^{\V(\Lambda)}$ also has a natural action as a group of automorphisms of $\Gamma$ (given by $(v,a)^x=(v,a+x(v))$, for $x\in \FF_2^{\V(\Lambda)}$ and $(v,a)\in \V(\Lambda)\times \FF_2$). 

Let $E_1$ be the $1$-eigenspace for $\Lambda$ over $\FF_2$. Note that $E_1\leq\FF_2^{\V(\Lambda)}$ and thus $E_1\leq\Aut(\Gamma)$.  It is not difficult to see that in fact $E_1\unlhd \Aut(\Gamma)$. Let $m=|\V(\Gamma)|=2|\V(\Lambda)|=32n^4$. As we saw in Section~\ref{sec:main}, we have $|E_1|\geq 2^{|\V(\Lambda)|/72}=2^{m/144}$.  Recall from Section~\ref{MKGraph} that $\Lambda$ admits a $2$-arc-regular group of automorphism $\tB$, as well as a  $1$-arc-regular subgroup $\tA\leq \tB$. 

First, let $G_1=\langle E_1,\tA\rangle=E_1\rtimes \tA\leq\Aut(\Gamma)$. Recall that $|\tA|=3\cdot 16n^4=3m/2$ and thus $|G_1|=|E_1||\tA|\geq 3m\cdot 2^{m/144-1}$. In particular, a vertex-stabiliser in $G_1$ has order at least $3\cdot 2^{m/144-1}$. Since the dimension of $E_1$ is more than one and $\Lambda$ is connected, $E_1$ must contain an element $x$ such that $x(v)=0$ and $x(u)=1$ for some $(u,v)\in\A(\Lambda)$. It easily follows that $\Gamma$ is $G_1$-arc-transitive. The stabiliser of the vertex $(v,0)$ in $G_1$ is $(E_1)_v\rtimes \tA_v$, where $(E_1)_v$ is the subspace (of codimension $1$) of $E_1$ consisting of vectors $x$ with $x(v)=0$. The local action at $(v,0)$ is $L:=((E_1)_v\rtimes \tA_v)^{\Gamma(v,0)}=(E_1)_v^{\Gamma(v,0)}\rtimes \tA_v^{\Gamma(v,0)}$. Clearly, $L$ has blocks of size $2$, and the kernel of the action on these blocks is $(E_1)_v^{\Gamma(v,0)}$, while the induced action on the blocks is isomorphic to $\tA_v^{\Gamma(v,0)}\cong \ZZ_3$. Now, we saw earlier that $(E_1)_v^{\Gamma(v,0)}$ is non-trivial. From the definition of the $1$-eigenspace, it follows that $(E_1)_v^{\Gamma(v,0)}$ must have order $4$. As remarked in the introduction, if a transitive group of degree $6$ admits an invariant partition with three parts of size $2$ such that the kernel of the action on the parts has order $4$, then it must be one of three listed in the introduction. By order considerations, we then see that  $L$ must be $A_4(6)$.

Next, let $G_2=\langle E_1,\tB\rangle=E_1\rtimes \tB\leq\Aut(\Gamma)$. We can repeat the argument in the last paragraph, with the significant differences being that a vertex-stabiliser in $G_2$ has order at least $3\cdot 2^{m/144}$, $L=(E_1)_v^{\Gamma(v,0)}\rtimes \tB_v^{\Gamma(v,0)}$ and that the action on the blocks of size $2$ is isomorphic to $\tB_v^{\Gamma(v,0)}\cong \Sym(3)$ and thus $|L|=24$. Let $\Lambda(v)=\{v_1,v_2,v_3\}$. Since  $\tB$ acts transitively on the $2$-arcs of $\Lambda$, there exists $\alpha\in \tB_{(v,v_3)}$ interchanging $v_1$ and $v_2$. The element of $L$ induced by $\alpha$ is the permutation $((v_1,0)\,(v_2,0))((v_1,1),(v_2,1))$. Further, since $(E_1)_v^{\Gamma(v,0)}$ is non-trivial, there exists $x\in (E_1)_v$ such that $x(v_1)=x(v_2)=1$ and $x(v_3)=0$. The element of $L$ induced by $x$ is the permutation $((v_1,0)\,(v_1,1))((v_2,0),(v_2,1))$. It is now easy to see that $L_{(v_3,0)}$ is isomorphic to the Klein group. Again, using the remark in the introduction, we see that $L=S_4(6d)$. It thus remains to deal with $S_4(6c)$. Before we can do this, we need the following.

\begin{lemma}
There exists  $\sigma\in \FF_2^{\V(\Lambda)}$ with the following properties:
\begin{itemize}
\item For all $b\in \tB$, we have $\sigma^b\sigma\in E_1$.
\item There exists an arc $(v,w)$ of $\Lambda$, such that $\sigma(v)=0$, $\sigma(w)=1$, and $\sigma(u)=0$ for all neighbours $u$ of $v$ other than $w$.
\end{itemize}
\end{lemma}

\begin{proof}
First, consider the analogous claim where $\Lambda$ is replaced by the complete graph $\K_4$ and $\tB$ is replaced by $\Aut(\K_4)$. Clearly, the analogous claim holds: simply consider $\sigma_{\K_4}\in \FF_2^{\V(\K_4)}$ such that $\sigma_{\K_4}(v)$ is $1$ for a unique vertex $v$ of $\K_4$, and $0$ for the others.

Next, note that there is a regular covering projection $\p:\Q_3\to \K_4$ (where $\Q_3$ denotes the graph of the $3$-cube), and that $\Aut(\K_4)$ lifts to $\Aut(\Q_3)$ along this projection. We now define $\sigma_{\Q_3}\in \FF_2^{\V(\Q_3)}$ using this projection in the most natural way: $\sigma_{\Q_3}(v)=\sigma_{\K_4}(\p(v))$. One can check that $\sigma_{\Q_3}$ is $1$ on a pair of antipodal vertices of $\Q_3$, and $0$ elsewhere. One can also check that $\sigma_{\Q_3}$ satisfies an analog of the claim, with $\Lambda$ replaced by $\Q_3$ and $\tB$ replaced by $\Aut(\Q_3)$. This can be checked by brute force but, more importantly, it immediately follows from the definition of $\sigma_{\Q_3}$ together with the fact that the analogous claim holds for $\K_4$.

This last remark allows us to repeat this procedure: there is a natural covering projection $\MK\to \Q_3$ such that $\Aut(\Q_3)$ lifts to $B$ and we can use this to define $\sigma_{\MK}$ as earlier. Again, $\sigma_{\MK}$ satisfies an analog of the claim, with $\Lambda$ replaced by $\MK$ and $\tB$ replaced by $B$. Finally, we repeat this procedure one last time, using the covering projection $\Lambda_n\to \MK$, along which $B$ lifts to $\tB$, to obtain $\sigma=\sigma_{\Lambda_n}$ defined in the natural way.
\end{proof}

Let $\sigma$, $v$ and $w$ be as in the claim and let $\tau\in \tB\setminus \tA$ such that $\tau$ fixes $v$ but not $w$. (Such an element exists since $\tB$ is $2$-arc-regular while $\tA$ is $1$-arc-regular.) Let $G_3=\langle G_1,\tau\sigma\rangle=\langle E_1, \tA , \tau\sigma\rangle\leq \Aut(\Gamma)$. We first show that $G_1$ is normal in $G_3$. It suffices to show that $\sigma\tau$ normalises $G_1$. Note that $E_1$ is centralised by $\sigma$ and normalised by $\tau$. Since $\tA$ is normal in $\tB$, it is normalised by $\tau$. It thus remains only to show that $\tA^\sigma\leq G_1$. Let $a\in \tA$. By definition of $\sigma$, we have $a^{-1}\sigma a \sigma=\sigma^a\sigma\in E_1$ and thus $a^\sigma=\sigma a\sigma\in aE_1\subseteq G_1$, as required. This shows that $G_1$ is normal in $G_3$. Note that $(\tau\sigma)^2=\tau^2\sigma^\tau\sigma\in \tA E_1\subseteq G_1$. It follows that $|G_3:G_1|\leq 2$.

Note that $\tau\sigma$ fixes $(v,0)$. In fact, it is easy to see that $\tau\sigma$ fixes a neighbour of $(v,0)$ (and thus in fact at least two neighbours). On the other hand, it acts as a $4$-cycle on the remaining $4$ neighbours of  $(v,0)$.  By the remark in the introduction, it follows that the local action of $G_3$ must be $S_4(6c)$ and thus $|G_3|=2|G_1|=|G_2|$. This concludes the proof of Theorem~\ref{theo:main2}.

\bigskip

\textbf{Acknowledgments}
The authors gratefully acknowledge the support of the mathematical research institute MATRIX that hosted the ``Tutte Centenary Retreat'' in November 2017 where part of the research was done. The first and second listed authors would also like to thank the University of Auckland for the hospitality during their visits. This work of the first listed author is supported in part by the Slovenian Research Agency (research program P1-0404 and research projects J1-1691, J1-1694, J1-1695,  J1-9110, N1-0102, N1-0140  and N1-0062).
The second listed author is partially supported by the Slovenian Research Agency, research project J1-1691 and research program P1-0294.

\section{Appendix: {\sc magma} code that checks that $x_1$ is a $1$-eigenvector for $\Lambda$}\label{Appendix}

\begin{verbatim}
/* The function Delta(n) defines the graph denoted by \Delta_n in the paper.
The labels of the vertices differ from those in the paper: 
The first coordinate, which is an element of the group R in the paper,
is an element of Z_8xZ_2 in the code. The correspondence is as follows:
(0,0) <-> a
(1,0) <-> ac
(2,0) <-> abc
(3,0) <-> abz
(4,0) <-> az
(5,0) <-> acz
(6,0) <-> abcz
(7,0) <-> ab
(0,1) <-> id
(1,1) <-> c
(2,1) <-> bc
(3,1) <-> bz
(4,1) <-> z
(5,1) <-> cz
(6,1) <-> bcz
(7,1) <-> b */


Delta:=function(n)
V:={<i,j,x,y,z,w>:i in [0..7], j in [0..1], x in [0..n-1], 
y in [0..n-1],z in [0..n-1],w in [0..n-1]};
E:={{<i,0,x,y,z,w>,<i,1,x,y,z,w>}:i in [0..7], x in [0..n-1], 
y in [0..n-1],z in [0..n-1],w in [0..n-1]};
E:=E join {{<0,1,x,y,z,w>,<1,1,(x+1) mod n,y,z,w>}:x in [0..n-1], 
y in [0..n-1],z in [0..n-1],w in [0..n-1]};
E:=E join {{<1,1,x,y,z,w>,<2,1,(x) mod n,(y+1) mod n,z,w>}:x in [0..n-1], 
y in [0..n-1],z in [0..n-1],w in [0..n-1]};
E:=E join {{<2,1,x,y,z,w>,<3,1,(x) mod n,y,(z+1) mod n,w>}:x in [0..n-1], 
y in [0..n-1],z in [0..n-1],w in [0..n-1]};
E:=E join {{<3,1,x,y,z,w>,<4,1,(x) mod n,y,z,(w+1) mod n>}:x in [0..n-1], 
y in [0..n-1],z in [0..n-1],w in [0..n-1]};
E:=E join {{<4,1,x,y,z,w>,<5,1,(x-1) mod n,y,z,w>}:x in [0..n-1], 
y in [0..n-1],z in [0..n-1],w in [0..n-1]};
E:=E join {{<5,1,x,y,z,w>,<6,1,(x) mod n,(y-1) mod n,z,w>}:x in [0..n-1], 
y in [0..n-1],z in [0..n-1],w in [0..n-1]};
E:=E join {{<6,1,x,y,z,w>,<7,1,(x) mod n,y,(z-1) mod n,w>}:x in [0..n-1], 
y in [0..n-1],z in [0..n-1],w in [0..n-1]};
E:=E join {{<7,1,x,y,z,w>,<0,1,(x) mod n,y,z,(w-1) mod n>}:x in [0..n-1], 
y in [0..n-1],z in [0..n-1],w in [0..n-1]};
E:=E join {{<i,0,x,y,z,w>,<(i+3) mod 8,0,x,y,z,w>}:i in [0..7], x in [0..n-1], 
y in [0..n-1],z in [0..n-1],w in [0..n-1]};
X:=Graph<V|E>;
return(X);
end function;

/* The function IsEigenvector(F,X) has as input a graph X and a set of vertices F.
 It tests whether the support vector of F is a 1-eigenvector of X over GF(2). */

IsEigenvector:=function(F,X)
A,V,E:=AutomorphismGroup(X);
n:=Degree(A);
Z2:=IntegerRing(2);
A2:=Matrix(Z2,AdjacencyMatrix(X));
AI:=A2+IdentityMatrix(Z2,n);
ZeroVector:=AI[1]+AI[1];
sum:=ZeroVector;
for s in F do
sum:=sum + AI[Position(V,s)];
end for;
result:=sum eq ZeroVector;
return(result);
end function;

/* The set S1 corresponds to the set S_1 defined in the paper. 
The only difference is with the labelling of vertices, as in the first comment above. */

S1:={<0, 0, 0, 0, 1, 2 >,<0, 0, 0, 1, 2, 2 >,<0, 0, 1, 1, 1, 1 >,
<0, 0, 1, 2, 2, 1 >,<0, 1, 1, 0, 1, 1 >,<0, 1, 1, 1, 1, 1 >,
<0, 1, 1, 1, 2, 1 >,<0, 1, 1, 2, 2, 1 >,<1, 0, 0, 0, 1, 2 >,
<1, 0, 1, 1, 1, 1 >,<1, 0, 1, 1, 2, 2 >,<1, 0, 2, 2, 2, 1 >,
<1, 1, 1, 0, 1, 1 >,<1, 1, 1, 0, 1, 2 >,<1, 1, 1, 1, 1, 1 >,
<1, 1, 1, 1, 2, 2 >,<1, 1, 2, 1, 1, 1 >,<1, 1, 2, 1, 2, 1 >,
<2, 0, 0, 1, 1, 2 >,<2, 0, 1, 1, 1, 1 >,<2, 0, 1, 2, 2, 2 >,
<2, 0, 2, 2, 2, 1 >,<2, 1, 0, 1, 1, 2 >,<2, 1, 1, 1, 1, 1 >,
<2, 1, 1, 1, 1, 2 >,<2, 1, 2, 1, 1, 1 >,<3, 0, 0, 1, 2, 2 >,
<3, 0, 1, 0, 1, 1 >,<3, 0, 1, 1, 1, 1 >,<3, 0, 1, 1, 2, 1 >,
<3, 0, 1, 1, 2, 2 >,<3, 0, 2, 2, 2, 1 >,<3, 1, 1, 1, 2, 1 >,
<3, 1, 1, 2, 2, 1 >,<3, 1, 2, 1, 2, 1 >,<3, 1, 2, 2, 2, 1 >,
<4, 0, 1, 0, 1, 2 >,<4, 0, 1, 1, 2, 2 >,<4, 0, 2, 1, 1, 1 >,
<4, 0, 2, 2, 2, 1 >,<4, 1, 1, 0, 1, 2 >,<4, 1, 1, 1, 1, 2 >,
<4, 1, 1, 1, 2, 2 >,<4, 1, 1, 2, 2, 2 >,<5, 0, 0, 0, 1, 2 >,
<5, 0, 1, 1, 1, 1 >,<5, 0, 1, 1, 2, 2 >,<5, 0, 2, 2, 2, 1 >,
<5, 1, 0, 1, 1, 2 >,<5, 1, 0, 1, 2, 2 >,<5, 1, 1, 1, 1, 1 >,
<5, 1, 1, 1, 2, 2 >,<5, 1, 1, 2, 2, 1 >,<5, 1, 1, 2, 2, 2 >,
<6, 0, 0, 0, 1, 2 >,<6, 0, 1, 0, 1, 1 >,<6, 0, 1, 1, 2, 2 >,
<6, 0, 2, 1, 2, 1 >,<6, 1, 0, 1, 2, 2 >,<6, 1, 1, 1, 2, 1 >,
<6, 1, 1, 1, 2, 2 >,<6, 1, 2, 1, 2, 1 >,<7, 0, 0, 0, 1, 2 >,
<7, 0, 1, 1, 1, 1 >,<7, 0, 1, 1, 1, 2 >,<7, 0, 1, 1, 2, 2 >,
<7, 0, 1, 2, 2, 2 >,<7, 0, 2, 1, 1, 1 >,<7, 1, 0, 0, 1, 2 >,
<7, 1, 0, 1, 1, 2 >,<7, 1, 1, 0, 1, 2 >,<7, 1, 1, 1, 1, 2 >};

/* We test if S1 is a 1-eigenvector for the graph Delta(4). */
IsEigenvector(S1,Delta(4));
\end{verbatim}

\end{document}